\documentclass[twoside,a4paper,reqno,11pt]{amsart} 
\usepackage[top=28mm,right=28mm,bottom=28mm,left=28mm]{geometry}

\usepackage{amsfonts, amsmath, amssymb, mathrsfs, bm, latexsym, stmaryrd, array, mathtools, bold-extra, enumerate,enumitem,color}

\newcommand{\N}{\mathbb{N}}

\newcommand{\im}{\mathrm{im}}

\renewcommand{\to}{\rightarrow}

\makeatletter
\newcommand{\imod}[1]{\allowbreak\mkern4mu({\operator@font mod}\,\,#1)}
\makeatother

\newtheorem{theorem}{Theorem} 
\newtheorem*{conj*}{Conjecture}

\newtheorem{thm}{Theorem}[section] 
\newtheorem{prop}[thm]{Proposition} 
\newtheorem{lem}[thm]{Lemma}

\theoremstyle{definition}
\newtheorem{rem}[thm]{Remark}

\setlist[enumerate]{itemsep=3pt,topsep=3pt}
\setlist[enumerate,1]{label=\textup{(\roman*)}}
\setlist[enumerate,2]{label=\textup{(\alph*)}}

\begin{document}

\author{Donna M. Testerman}
\address{D.~M.~Testerman, Ecole Polytechnique F\'ed\'eral de Lausanne,
Institute of Mathematics, Station 8, CH-1015 Lausanne, Switzerland}
\email{donna.testerman@epfl.ch}
 
\author{Adam R. Thomas}
\address{A.~R.~Thomas, Mathematics Institute, Zeeman Building, University of Warwick, Coventry CV4 7AL, UK}
\email{adam.r.thomas@warwick.ac.uk}

\setlength{\parindent}{0pt}
\setlength{\parskip}{6pt}

\title{Epimorphic subgroups of simple algebraic groups}

\begin{abstract}
A morphism of linear algebraic groups $\phi:K\rightarrow G$ is called an epimorphism if it admits right cancellation. A subgroup $H\leq G$ is epimorphic if the inclusion map is an epimorphism. For $G$ a simple algebraic group over an algebraically closed field of arbitrary characteristic we construct epimorphic subgroups of bounded dimension (at most five).  
\end{abstract}

\maketitle

\section{Introduction}\label{sec:intro}
Let $\phi:K \rightarrow G$ be a morphism of linear algebraic groups defined over an algebraically closed field. The map $\phi$ is an epimorphism if it admits right cancellation, i.e. whenever $\psi_1\circ\phi=\psi_2\circ\phi$ for
morphisms $\psi_1,\psi_2$ we have $\psi_1=\psi_2$. 
An \emph{epimorphic subgroup} $H\leq G$ is a subgroup for which the inclusion map is an epimorphism. 

It follows from \cite{BB1} that the maximal dimensional epimorphic subgroups are maximal parabolic subgroups. The current state of knowledge on minimal dimensional epimorphic subgroups is less satisfactory, although it does follow from \textit{ibid.} that such subgroups are solvable. When $G$ is a simple algebraic group of rank at least $2$, the minimal dimension of an epimorphic subgroup appears to be $3$. If the ground field is $\mathbb C$, this is shown to hold by Bien and Borel in \cite{BB1}. In \cite{SimTe}, Simion and the first author consider rank 2 groups over fields of positive characteristic, establishing the same lower bound of 3, and showing it to be sharp. Our main result yields proper epimorphic subgroups of bounded dimension (at most five) in simple algebraic groups defined over algebraically closed fields of positive characteristic. 

\begin{theorem} \label{mainthm} Let $G$ be a simple algebraic group defined over an algebraically closed field of characteristic $p > 0$. Then there exists a proper epimorphic subgroup $H$ of $G$ with $\dim H\leq 5$.
\end{theorem}

In many cases we prove the existence of an epimorphic subgroup of dimension less than five; the precise dimensions in these cases are as in the following table. 

\begin{table}[ht]
\begin{tabular}{| c | c |}
\hline
$\dim H$ & $G$ \\
\hline
2 & $A_1$ \\
3 & $A_2$, $B_2$, $B_l$ $(p = 2)$, $C_l$, $D_l$ ($l$ even or $l \equiv 3 \pmod 6$), $G_2$, $F_4$, $E_6$ $(p \neq 2)$, $E_7$, $E_8$  \\
4 & $A_l$ $(l \geq 3)$, $B_l$ $(l \geq 3, p \neq 2)$, $E_6$ $(p = 2)$ \\
\hline
\end{tabular}
\label{tab:dimensions}
\end{table}


It is often the case when considering the subgroup structure of algebraic groups that restricting attention to ``large enough'' characteristics yields stronger results akin to those in characteristic $0$. We are able to make this concrete in the next result.

\begin{theorem} \label{mainthm2} Let $G$ be a simple algebraic group defined over an algebraically closed field of characteristic $p \geq h(G)$, the Coxeter number of $G$. Then there exists an epimorphic subgroup $H$ of $G$ of dimension three.
\end{theorem}

A study of epimorphic subgroups in linear algebraic groups was initiated by Bien and Borel in \cite{BB1,BB2} where they establish criteria for recognising epimorphic subgroups. These results were extended in \cite{Brion} to the setting of group schemes of finite type over any field.  

A strong motivation for Bien and Borel came from algebraic geometry. Together with Koll\'ar, in \cite{BBK} they used epimorphic subgroups to construct certain generically rationally connected homogeneous spaces. Their construction depends on having an epimorphic subgroup of a specific form (see \cite[Propositions~4.1 and 4.2]{BBK}). In all cases, our construction does in fact yield an epimorphic subgroup $H$ such that $G/H$ is a generically rationally connected homogeneous space for $G$.

\section{Preliminaries} \label{sec:prelims}

\subsection{Notation}

Let $k$ be an algebraically closed field of characteristic $p \geq 0$ and $k^*$ be its multiplicative group of nonzero elements. Throughout the paper all algebraic groups are defined over $k$ and all subgroups are assumed to be closed. An \textit{overgroup} of a subgroup $X \leq G$ is defined to be a proper subgroup $H\leq G$ such that that $X \leq H$. 

Let $X$ be a semisimple group. We let $\Phi$ be the root system of $X$ with respect to some maximal torus $T_X$ and Borel subgroup $B_X$. For a simple algebraic group $G$ we will simply write $T,B$ for such subgroups. Let $\Pi = \{ \alpha_1, \ldots, \alpha_l \} \subseteq \Phi$ be a base of simple roots and let $\{\lambda_1,\ldots,\lambda_{l}\}$ be the corresponding fundamental dominant weights, with the Bourbaki ordering \cite[Ch.\ VI, Planches I-IX]{Bourb4-6}. The notation $a_1 a_2 \ldots a_l$ will indicate a root $\sum a_{i}\alpha_{i}$ or a weight $\sum a_{i}\lambda_{i}$ (context will prevent ambiguity). Root elements and groups with respect to these fixed choices are denoted $x_\gamma(t)$, $U_{\gamma}$, respectively for $\gamma \in \Phi$ and $t \in k$. The highest root of $\Phi$ is denoted $\alpha_0$. Whenever we say subsystem subgroup we mean a semisimple subgroup normalised by the maximal torus $T$ of $G$. We use $L(X)$ for the Lie algebra of $X$ and fix a choice of root elements $e_\gamma \in L(X)$ for $\gamma \in \Phi$. 

For a dominant weight $\lambda$ we let $V_{X}(\lambda)$ denote the irreducible $X$-module of highest weight $\lambda$. When $X$ is clear from the context we simply write $\lambda$ for the module; in particular $0$ denotes the trivial irreducible $X$-module. The corresponding Weyl module for $X$ is denoted $W(\lambda)$ or $W_{X}(\lambda)$, and the tilting module is denoted $T(\lambda)$ or $T_{X}(\lambda)$. When $X$ is of type $A_1$ there is a single fundamental dominant weight $\lambda$ and we use the integer $n$ for the weight $n \lambda$. If $X = X_1 X_2 \cdots X_r$ is a commuting product of simple algebraic groups then $(V_1, \ldots, V_r)$ denotes the $X$-module $V_1 \otimes \cdots \otimes V_r$, where $V_i$ is an irreducible $X_i$-module for each $i$.

Suppose $p > 0$. Let $F{:} \ X \to X$ be a Frobenius endomorphism which acts on root elements via $x_{\alpha}(t) \mapsto x_{\alpha}(t^p)$ for $\alpha \in \Phi$, $t \in k$. If $V$ is an $X$-module afforded by a representation $\rho : X \rightarrow GL(V)$ then $V^{[p^r]}$ denotes the module afforded by $\rho^{[r]} := \rho \circ F^r$. Let $M_1, \ldots, M_k$ be $X$-modules. Then $V = M_1 | \ldots | M_r$ denotes an $X$-module with socle series $V = V_{1} > V_{2} > \cdots > V_{r+1} = \{0\}$, so that $\textrm{Socle}(V/V_{i+1}) = V_{i}/V_{i+1} \cong M_{i}$ for $1 \le i \le r$. 

Finally, the natural numbers, denoted $\N$, start from $1$. 








\subsection{General Strategy and Proof of Theorem \ref{mainthm2}}

The following result allows us to choose the isogeny type of $G$ when proving Theorems \ref{mainthm} and \ref{mainthm2}. 

\begin{lem}\cite[Proposition~6]{SimTe}\label{lem:isogenies}
Let $\phi:G_1\rightarrow G_2$ be a surjective homomorphism of algebraic groups. 
\begin{enumerate}
\item If $H_1$ is an epimorphic subgroup of $G_1$ then $\phi(H_1)$ is epimorphic in $G_2$,
\item If $H_2$ is an epimorphic subgroup of $G_2$ then $\phi^{-1}(H_2)$ is epimorphic in $G_1$.
\end{enumerate}
\end{lem}

Theorem \ref{mainthm} and \ref{mainthm2} are immediate for $G$ of type $A_1$ since a Borel subgroup of $G$ is a $2$-dimensional epimorphic subgroup. The main theorems hold for $G$ of rank two courtesy of \cite[Theorem~1]{SimTe}. Thus, we assume that $G$ is of rank at least three. 

Our main tool for establishing the existence of epimorphic subgroups is the following proposition.

\begin{prop}\label{prop:BBmethod} Let $J$ be a subgroup of type $A_1$ in $G$ and let $Y$ be a subgroup of $G$ normalised by a Borel subgroup $B_J$ of $J$. If $\langle J, Y\rangle = G$, then $B_JY$ is epimorphic in $G$.
\end{prop}

\begin{proof} Let $H = B_JY$, $X_1 = J$ and $X_2 = Y$. Since $X_1\cap H$ contains $B_J$, it is epimorphic in $J$. Clearly $X_2\cap H=Y$ is epimorphic in $Y$ and so Result (e) in \cite[Section~2]{BB1} implies that $H$ is epimorphic in $G$.\end{proof}

We can now prove Theorem \ref{mainthm2}, illustrating our use of Proposition \ref{prop:BBmethod}.

\begin{proof}[Proof of Theorem \ref{mainthm2}]
When $p\geq h$ the regular unipotent elements in $G$ have order $p$ (see \cite[Theorem~0.4]{TestA1}). Therefore, \cite[Theorem~1]{SeitzgoodA1} shows there exists a unique conjugacy class of subgroups of $G$ of type $A_1$ which contain regular unipotent elements; fix a subgroup
$J\leq G$ in this class. By \cite[Theorem~1.2]{TeZ}, $J$ does not lie in a proper parabolic subgroup of $G$ and the reductive overgroups of $J$ are given in Table 1, {\it ibid.} One sees that for $G$ of type $B_l$ ($l\geq 4$), $C_l$, $F_4$, $E_7$ or $E_8$, the subgroup $J$ is a maximal subgroup of $G$. We may assume that a Borel subgroup of $J$ is contained in $B$ and therefore normalises $Y = U_{\alpha_0}$. The subgroup $Y$ does not lie in $J$ as the unique class of nonidentity unipotent elements represented in $J$ is the class of regular elements; hence we have $\langle J, Y\rangle = G$ and Proposition~\ref{prop:BBmethod} gives the result.
 
We treat the remaining cases from \cite[Table~1]{TeZ} one-by-one. We exhibit a $1$-dimensional unipotent group $Y$, normalised by a Borel subgroup of $J$ with the property that $Y$ is not contained in any maximal subgroup $M$ of $G$ which contains $J$. Thus we deduce that $\langle J,Y\rangle = G$ and conclude as in the previous paragraph.

Let $G$ be of type $A_l$ with $l\geq 3$ odd. The unique maximal overgroup of $J$ is a subgroup $M$ of type $C_{(l+1)/2}$ which acts completely reducibly on $L(G)$, with summands $L(M)$ and $V_M(\lambda_2)$ (see \cite[Table~2]{Lubeck}). Up to conjugacy, we may assume that $M$ is the fixed point subgroup of a standard graph automorphism of $G$. Then setting $\gamma = \alpha_1+\cdots+\alpha_{l-1}$ and $\delta = \alpha_2+\cdots+\alpha_l$, and for an appropriate choice of $c\in\{1,-1\}$, the subgroup $Y = \{x_\gamma(t)x_\delta(ct) \mid t\in k\}$ is normalised by a Borel subgroup of $M$ and hence by a Borel subgroup of $J$. (Note that $v\in L(Y)$, $v\ne 0$, affords a maximal vector of the summand $V_M(\lambda_2)$.) Since $Y$ does not lie in $M$, we have that $\langle J,Y\rangle = G$.

For $G$ of type $A_l$ with $l\geq 4$ even, the unique maximal subgroup of $G$ containing $J$ is a subgroup $M$ of type $B_{l/2}$. For an appropriate choice of $M$, we may take $Y$ to be the root subgroup $U_{\alpha_0}$ and we conclude as in the previous case.

Let $G$ be of type $D_l$ with $l \geq 4$. The subgroup $J$ lies in a maximal subgroup $M_1$ of type $B_{l-1}$. By Lemma \ref{lem:isogenies} we may assume that $G = \text{SO}(W)$. Then $M_1$ is the stabiliser of a non-degenerate $1$-dimensional subspace of $W$. Since $J$ acts on $W$ as $V_J(2l-2) + 0$, it follows that $J$ is contained in a unique maximal subgroup of type $B_{l-1}$ which stabilises a $1$-dimensional subspace of $W$. If $l>4$, this is the unique maximal overgroup of $J$ and $M$ acts completely reducibly on $L(G)$ with summands $L(M_1)$ and $V_{M_1}(\lambda_1)$ (see \cite[Table~2]{Lubeck}). By choosing $J$ appropriately up to conjugacy, we may assume that $M_1$ is the fixed point subgroup of a standard involutary graph automorphism of $G$ which interchanges the root groups $U_{\alpha_{l-1}}$ and $U_{\alpha_l}$. Letting $\gamma = \alpha_1+\cdots+\alpha_{l-1}$ and $\delta = \alpha_1+\cdots+\alpha_{l-2}+\alpha_l$, the subgroup $Y = \{x_\gamma(t)x_\delta(-t) \mid t\in k\}$ is normalised by a Borel subgroup of $M$ and hence by a Borel subgroup of $J$. (Note that $0 \neq v\in L(Y)$ affords a maximal vector of the summand $V_{M_1}(\lambda_1)$.) Since $Y$ does not lie in $M_1$, we have that $\langle J,Y\rangle = G$.

When $l=4$, there are three classes of maximal subgroups containing a conjugate of $J$, all of which are of type $B_3$ and are conjugate in the automorphism group of $G$. Since $J$ is contained in a unique maximal subgroup $M_1$ of type $B_3$ acting as $V_{B_3}(100) + 0$ on $W$, it follows that $J$ is contained in exactly three maximal subgroups, say $M_1, M_2, M_3$, one of each class ($M_2$ and $M_3$ act irreducibly on $W$). These three maximal subgroups are each the fixed point subgroup of the three involutary graph automorphisms of $G$ which swap the root groups $U_{\alpha_{3}}$ and $U_{\alpha_4}$, $U_{\alpha_{4}}$ and $U_{\alpha_1}$, and $U_{\alpha_{1}}$ and $U_{\alpha_3}$, respectively. Now it is immediate that $Y$, as defined in the previous paragraph, is not contained in any of the three maximal overgroups of $J$ so again $\langle J,Y\rangle = G$.

Now we treat the cases of bounded rank. Let $G$ be of type $B_3$. Then $J$ is contained in a maximal subgroup $M \leq G$ of type $G_2$. One calculates that $L(G)\downarrow M = L(M) + V_M(10)$ and restricting to $J$ we have $L(G)\downarrow J = L(M) \downarrow J + 6$. By \cite[Table~10.1]{LSexc}, the action of $J$ on $L(M)$ is $10 + 2$ if $p > 7$ or $2|10|2$ if $p=7$. In particular, both $M$ and $J$ fix a unique $7$-dimensional subspace of $L(G)$. Since $M$ is a maximal subgroup of $G$ it follows that $M$ is the unique maximal connected overgroup of $J$. Choosing $J$ such that a Borel subgroup of $J$ is contained in $B$, the highest weight vector of the $V_M(10)$ summand of $L(G)$ is of the form $ce_{111} + de_{012}$, for some $c,d\in k$. Letting $Y = \{x_{111}(ct)x_{012}(dt) \mid t\in k\}$, we have that $Y$ is normalised by a Borel subgroup of $J$ and is not contained in $M_1$, so that $\langle J,Y\rangle = G$. 

Finally, suppose $G$ is of type $E_6$ so that $J$ is contained in a maximal subgroup $M$ of type $F_4$. Using \cite[Table~10.2]{LSexc} we see that $J$ and $M$ fix a unique $1$-dimensional subspace of $V_{G}(\lambda_1)$; thus $M$ is the unique maximal overgroup of $J$. By choosing $J$ appropriately up to conjugacy, we may assume that $M$ is the fixed point subgroup of the standard graph automorphism of $G$. One checks that $Y = \{x_{111221}(t)x_{112211}(-t) \mid t\in k\}$ is normalised by a Borel subgroup of $M$, hence also of $J$ but that $Y$ is not a subgroup of $M$. Thus $\langle J,Y\rangle = G$.
\end{proof}

\section{Proof of Theorem \ref{mainthm} for $G$ of classical type} \label{sec:classical}

In this section $G$ is assumed to be of classical type, $p > 0$ and we recall that the rank of $G$ is assumed to be at least $3$. Furthermore, Lemma~\ref{lem:isogenies} allows us to assume that $G = \textrm{SL}(W), \textrm{Sp}(W)$ or $\textrm{SO}(W)$ for some vector space $W$ of dimension at least $4$. 

\begin{lem}\label{lem:torus} Let $T_r$ be an $r$-dimensional torus with $T_r = \{\chi_1(t_1)\cdots\chi_r(t_r) \mid t_i\in k^*\}$ for some $\chi_1, \dots, \chi_r\in Y(T_r)$. For $a\in\N$ and $t \in k^*$, let 
\[ s_a(t) = \chi_1(t)\chi_2(t^{p^a})\cdots\chi_r(t^{p^{a(r-1)}}),\] 
if $r$ is even let \[s'_a(t) = \chi_1(t^{1+p^a})\chi_2(t^{-1+p^a}) \chi_3(t^{p^{2a}+p^{3a}})\chi_4(t^{-p^{2a}+p^{3a}}) \cdots  \] \[ \chi_{r-1}(t^{p^{a(r-2)} + p^{a(r-1)} }) \chi_r(t^{-p^{a(r-2)} + p^{a(r-1)} }),\] 
and if $r$ is odd let
\[s''_a(t) = \chi_1(t^{1+p^a})\chi_2(t^{-1+p^a}) \chi_3(t^{p^{2a}+p^{3a}})\chi_4(t^{-p^{2a}+p^{3a}}) \cdots  \] \[ \chi_{r-2}(t^{p^{a(r-3)} + p^{a(r-2)} }) \chi_{r-1}(t^{-p^{a(r-3)} + p^{a(r-2)} }) \chi_r(t^{p^{a(r-1)}}).\] 
If a subtorus $S\leq T_r$ contains $\im(s_a)$ for infinitely many $a$, or contains $\im(s'_a)$ for infinitely many $a$, or contains $\im(s''_a)$ for infinitely many $a$, then $S = T_r$.
\end{lem}

\begin{proof}
By \cite[Prop. 3.8]{MaTe}, $X(S) = X(T_r)/S^\perp$, where $S^\perp = \{\gamma\in X(T_r) \mid \gamma(s) = 1$ for all $s\in S\}$. Let $\gamma\in S^\perp$. We show that $\gamma = 0$ so that $S = T_r$. 

Suppose first that $S$ contains $\im(s_a)$ for infinitely many $a$ and choose distinct $a_1, \ldots, a_{r} \in \N$ such that $\im(s_{a_i})\leq S$ for $1\leq i\leq r$. Then $\gamma(s_{a_i}(t)) = 1$ for all $t\in k^*$, i.e. \[\sum\limits_{j=1}^{r} p^{a_i(j-1)} \langle \gamma,\chi_j\rangle  = 0.\] The corresponding matrix $(p^{a_j(i-1)})_{i,j}$ is a Vandermonde matrix, whose determinant is nonzero. So the system of equations only has the trivial solution, $\langle \gamma,\chi_i\rangle=0$ for all $1\leq i \leq r$. Hence $\gamma = 0$.

The case where $S$ contains $\im(s'_a)$ for infinitely many $a$ is similar but with one change. This time, there exist distinct $a_1,\dots,a_{r}\in \N$ such that for each $1 \leq i \leq r$ we have
\[ \sum\limits_{j=1}^{\frac{r}{2}} \left(p^{(2j-2)a_i} + p^{(2j-1)a_i}\right) \langle \gamma,\chi_{2j-1} \rangle + \sum\limits_{j=1}^{\frac{r}{2}} \left(-p^{(2j-2)a_i} + p^{(2j-1)a_i}\right) \langle \gamma,\chi_{2j} \rangle = 0.\]

The corresponding $r \times r$ matrix of coefficients of $\langle \gamma,\chi_{j} \rangle$ is not a Vandermonde matrix this time. Instead, for $j =1, \ldots, r/2$ we implement the column operations replacing the $2j$th column ($c_{2j}$) with the sum of the $2j$th and $(2j-1)$st columns ($c_{2j} + c_{2j-1}$) and then replace column $(c_{2j-1})$ with column $(2c_{2j-1} - c_{2j})$. The resulting matrix is the Vandermonde matrix from above with every entry doubled, and thus has nonzero determinant. The $\im(s''_a)$ case is entirely similar.  
\end{proof}

\begin{prop}\label{prop:LSTplusSeitz} 
Let $M$ be a connected subgroup of $G$ acting irreducibly and tensor indecomposably on $W$. If $M$ preserves a nondegenerate bilinear form $f$ on $W$ or a nondegenerate quadratic form $q$ when $p=2$, let $H$ be the full isometry group $\textrm{Isom}(f)$ or $\textrm{Isom}(q)$. Suppose $M$ is a maximal subgroup of $H$ and $\textrm{rank}(M) > \frac{\textrm{rank}(H)}{2}$ or $H$ is of type $A_{2r}$ and $\textrm{rank}(M) = r$. Then $M$ is of type $G_2$ and $H$ is of type $B_3$ with $p \neq 2$, or $H$ is of type $C_3$ with $p = 2$.
\end{prop}

\begin{proof}
Since $M$ acts irreducibly and tensor indecomposably on $W$ (of dimension at least $3$) it follows that $M$ is simple. 

First suppose $\textrm{rank}(M) > \frac{\textrm{rank}(H)}{2}$. By \cite[Theorem~1]{LSTlargerank}, either $M$ is as in the conclusion of the proposition or there exists a maximal rank subsystem subgroup $K \leq G$ containing $M$ (note that the case $A_2 < B_3$ in \textit{loc. cit.} does not correspond to a maximal embedding). In the second case, since $K$ contains $M$, it also acts irreducibly and tensor indecomposably on $W$. This is ruled out by \cite[Theorem~4.1]{Seitzclass}.    

Now $G$ is of type $A_{2r}$ and $M$ has rank $r$ and an irreducible module $V$ of dimension $2r+1$. If $r > 11$, \cite[Theorem~5.1 and Table~2]{Lubeck} show that $M$ is of type $B_r$, $p \neq 2$ and $V$ is the natural module of dimension $2r+1$. However, $M = \textrm{Isom}(f)$ for a nondegenerate form $f$ on $V$ and so $M$ is not a maximal subgroup of $H$. If $r \leq 11$ we use Theorem~4.1 and the tables in Appendix~A of \textit{ibid}. These allow us to conclude that there are no further possibilities for $(M,V)$.    
\end{proof}

\begin{prop}\label{prop:classicaltool}
For each $a \in \N$, let $J_a \leq G$ be a subgroup of type $A_1$ such that $W\downarrow J_a = W_1 \oplus \cdots \oplus W_r$ as in one of the following cases. 
\begin{enumerate}
\item $G$ is of type $A_{2l}$, $p \neq 2$, $l=r$ and $W_i$ is irreducible of highest weight $p^{a(i-1)}$ for $1 \leq i \leq r-1$  and $W_r$ is irreducible of highest weight $2p^{a(r-1)}$,
\item $G$ is of type $B_{l}$, $p \neq 2$, $l=2r-1$ and $W_i$ is irreducible of highest weight $p^{a(2i-2)} +p^{a(2i-1)}$ for $1 \leq i \leq r-1$ and $W_r$ is irreducible of highest weight $2p^{a(2r-2)}$,
\item $G$ is of type $C_l$, $l=r$ and $W_i$ is irreducible of highest weight $p^{a(i-1)}$ for $1 \leq i \leq r$, or
\item $G$ is of type $D_{l}$, $l=2r$ and  $W_i$ is irreducible of highest weight $p^{a(2i-2)} +p^{a(2i-1)}$ for $1 \leq i \leq r$. 
\end{enumerate}
Furthermore, suppose that for each $a \in \N$ there exists a connected subgroup $X_a$ acting irreducibly on $W$ such that $J_a \leq X_a$. Then there exists $b \in \N$ such that $X_b = G$. 
\end{prop}

\begin{proof}
Assume, for a contradiction, that $X_a$ is a proper subgroup of $G$ for all $a \in \N$. For $a \in \N$ let $M_a$ be a maximal connected overgroup of $X_a$. Any maximal connected overgroup of $X_a$ is reductive since $X_a$ acts irreducibly on $W$. By \cite[Theorem~3]{Seitzclass} and \cite[Theorem~1]{LiSe04}, there are only finitely many conjugacy classes of maximal connected reductive subgroups of $G$. Therefore, there exists an infinite set $I \subset \N$ such that all $M_i$ and $M_j$ are conjugate for $i,j \in I$. Therefore, there exists $j \in I$ such that $M = M_j$ contains $X_i^{g_i}$ for some $g_i \in G$ and all $i \in I$.

We claim that $M$ has rank at least $l$. With this claim we can quickly deduce a contradiction as follows. Since $X_j$ is contained in $M$ it follows that $M$ acts irreducibly and tensor indecomposably on $W$. In cases (i) and (iii) the subgroups $J_a$ do not preserve a nondegenerate bilinear, respectively quadratic form on $W$, and so in all cases $M$ is not contained in any proper isometry subgroup of $G$. Now apply Proposition~\ref{prop:LSTplusSeitz}. The two exceptions listed cannot occur in our situation since $G_2$ does not have rank $3$ and so no such $M$ can exist. 

It remains to prove the claim. Let $T_M \leq M$ be a maximal torus of $M$. Conjugating by elements of $M$, we may assume that a maximal torus of each $J_i^{g_i}$ is contained in $T_M$ for all $i \in I$. Fix an ordered basis of $W$ so that $T_M$ lies in the set of diagonal matrices with respect to this basis. 

Suppose we are in case (i). Let $S_i \leq J_i^{g_i}$ be a maximal torus of $J_i^{g_i}$ contained in $T_M$ for all $i \in I$. By our choices above, we know that $S_i$ is a subgroup of the diagonal torus $T_D$ of ${\rm SL}(W)$. With only finitely many distinct permutations of the fixed ordered basis of $W$, we may assume that for some infinite subset $A \subseteq I$ we have $$S_a = \{{\rm diag}(c,c^{-1},c^{p^a}, c^{-p^a}, \dots, c^{p^{(r-2)a}},c^{-p^{(r-2)a}}, c^{2p^{a(r-1)}},1, c^{-2p^{a(r-1)}}) \mid c\in k^*\}$$ for each $a \in A$. The torus $T_D$ has cocharacters $\chi_i: k^* \rightarrow T_D$ where $\chi_i(c)$ for $1 \leq i \leq r-1$ is the diagonal matrix with $c$ in the $(2i-1)$st position, $c^{-1}$ in the $2i$th position and $1$s otherwise, and $\chi_r(c)$ is the diagonal matrix with $c^2$ in the $(r-2)$nd position, $c^{-2}$ in the $r$th position and $1$s otherwise. Applying  Lemma~\ref{lem:torus}, we see that $S = \langle S_a \mid a \in A \rangle$ is a torus of dimension $r=l$ and hence $T_M$ has dimension at least $l$. Case (iii) follows in exactly the same way.  

Now suppose we are in case (ii). Again after reordering the basis of $W$, there exists an infinite subset $A \subseteq I$ such that 
\[ S'_a = \{{\rm diag}(c^{1+p^a},c^{-1-p^a},c^{-1+p^a}, c^{1-p^a},\ldots, c^{p^{(r-4)a}+p^{(r-3)a}}, c^{-p^{(r-4)a}-p^{(r-3)a}},\] \[c^{-p^{(r-4)a}+p^{(r-3)a}}, c^{p^{(r-4)a}-p^{(r-3)a}}, c^{2p^{a(2r-2)}},1, c^{-2p^{a(2r-2)}}) \mid c\in k^*\},\] for each $a \in A$. 
Lemma~\ref{lem:torus} shows that $S = \langle S'_a \mid a \in A \rangle$ is a torus of dimension $l$. Case (iv) follows in exactly the same way.
\end{proof}

In the next sections we deal with all classical types case-by-case. We note that the case when $G$ is of type $B_l$, $p=2$ follows from the $C_l$, $p=2$ case using Lemma \ref{lem:isogenies}. 

\subsection{$G$ of type $C_l$}\label{sec:C}

Let $K \leq G$ be a subsystem subgroup of type $A_1^l$ arising from the natural subgroup ${\rm Sp}_2 \times \cdots \times {\rm Sp}_2\leq {\rm Sp}_{2l}$ (recall that our subsystem subgroups are always chosen such that they contain $T$). Fix $a \in \N$. Let $J_a$ be a subgroup of type $A_1$ diagonally embedded in $K$ with twists $q_1 = 1, q_2=p^a,\dots,q_l = p^{a(l-1)}$ and chosen such that a Borel subgroup $B_a \leq  J_a$ (resp. a maximal torus $T_a \leq  B_a \leq J_a$) is contained in $B$ (resp. $T$). Then $W\downarrow J_a = \oplus_{i=1}^l W_i$, where $W_i$ is irreducible of highest weight $q_i$, for $1\leq i\leq l$. Note that 
\[ L(G)\downarrow J_a = L(K) \downarrow J_a \oplus(\oplus_{1\leq i<j\leq l} (q_i+q_j)).\] Furthermore, the $T_a$-weight spaces of $L(G)$ corresponding to nonzero weights are $1$-dimensional and afforded by root vectors with respect to $T$. Let $\gamma_i\in\Phi$ be such that $\gamma_i\downarrow T_a = q_{i}+q_{i+1}$, for $1\leq i\leq l-1$. Since $q_i+q_{i+1}+q_j+q_{j+1}$ is not a $T_a$-weight of $L(G)$, we have that $\gamma_i+\gamma_j\not\in\Phi$ and so the corresponding $T$-root subgroups commute. Therefore $Y_a = \{ x_{\gamma_1}(t^{q_1}) \cdots x_{\gamma_{l-1}}(t^{q_{l-1}}) \mid  t\in k\}$ is a 1-dimensional unipotent group normalised by $T_a$, acting with weight $1+p^a$. Similar weight based arguments show that $Y_a$ is normalised by $B_a$.

Let $X_a=\langle J_a,Y_a\rangle$. We claim that $X_a$ acts irreducibly on $W$. Since the $W_i$ are pairwise nonisomorphic $J_a$-modules, any $X_a$-submodule of $W$ is a sum of some of the $W_i$. Note that the $\gamma_i$ are short roots, so that the root elements $x_{\gamma_i}(t)$ $(t \neq 0)$ act on $W$ with Jordan blocks $J_2^2+J_1^{2l-4}$. The $T_a$-weights of $W_j$ are $q_j, -q_j$ and so it follows that $W_j$ is fixed by $x_{\gamma_i}(t)$ for all $j\not\in\{i,i-1\}$. Furthermore, the projection of $Y_a(W_1)$ to $W_2$ is nonzero, the projection of $Y_a(W_l)$ to $W_{l-1}$ is nonzero and the projections of $Y_a(W_i)$ for $1 < i < l$ to $W_{i-1}$ and $W_{i+1}$ are both nonzero. Hence $X_a$ stabilises no proper subspace of $W$, establishing the claim.

We are now in a position to apply Proposition \ref{prop:classicaltool}. This yields $b \in \N$ such that $X_b = G$. Therefore, $B_b Y_b$ is a $3$-dimensional epimorphic subgroup of $G$ by Proposition~\ref{prop:BBmethod}. 

\subsection{$G$ of type $D_l$ with $l$ even}\label{sec:Deven}

Let $K \leq G$ be a subsystem subgroup of type $A_1^l$ arising from the natural subgroup ${\rm SO}_4 \times \cdots \times {\rm SO}_4\leq {\rm SO}_{2l}$. Fix $a \in \N$. Let $J_a$ be a subgroup of type $A_1$ diagonally embedded in $K$ with twists $q_1 = 1, q_2=p^a,\dots,q_l = p^{a(l-1)}$ and chosen such that a Borel subgroup $B_a \leq J_a$ (resp. a maximal torus $T_a \leq B_a \leq J_a$) is contained in $B$ (resp. $T$). Then $W \downarrow J_a = \oplus_{i=1}^{l/2} W_i$, where $W_i$ is irreducible of highest weight $q_{2i-1}+q_{2i}$ for $1\leq i \leq l/2$. Note that \[L(G)\downarrow J_a = L(K)\downarrow J_a \oplus \left(\bigoplus_{1\leq i < j \leq l/2} (q_{2i-1}+q_{2i}+q_{2j-1}+q_{2j})\right).\] Furthermore, the $T_a$-weight spaces of $L(G)$ corresponding to nonzero weights are $1$-dimensional and afforded by root vectors with respect to $T$. Let $\gamma_{i,j}\in\Phi$ be such that $\gamma_{i,j}\downarrow T_a = q_{2i-1}+q_{2i}+q_{2j-1}+q_{2j}$, for $1\leq i < j \leq l/2$. Note that $\gamma_{i,j}+\gamma_{m,n}\not\in\Phi$ for $1\leq m < n \leq l/2$ since $q_{2i-1}+q_{2i}+q_{2j-1}+q_{2j} + q_{2m-1}+q_{2m}+q_{2n-1}+q_{2n}$ is not a $T_a$-weight of $L(G)$. Thus $U_{\gamma_{i,j}}$ and $U_{\gamma_{m,n}}$ commute. 

Then $Y_a=\{x_{\gamma_{1,3}}(t^{q_1}) x_{\gamma_{3,5}}(t^{q_3})\cdots x_{\gamma_{l-3,l-1}}(t^{q_{l-3}}) \mid t\in k \}$ is a $1$-dimensional unipotent group normalised by a Borel subgroup of $J_a$. Let $X_a=\langle J_a, Y_a\rangle$; we claim that $X_a$ acts irreducibly on $W$. The nontrivial root elements $x_{\gamma_{i,j}}(t)$ fix all weight vectors in $W$ of weight different from $-q_{2i-1} - q_{2i}$ and $-q_{2j-1} - q_{2j}$. Since root elements of $G$ act on $W$ with Jordan blocks $J_2^2 + J_1^{2l-4}$, it follows that the endomorphisms $x_{\gamma_{i,j}}(t) - \text{id}_W$ $(t \neq 0)$ act injectively on the $2$-dimensional space spanned by weight vectors of weights $-q_{2i-1} - q_{2i}$ and $-q_{2j-1} - q_{2j}$. Combining these observations we see that the projection of $Y_a(W_i)$ to $W_{i-1}$ is nonzero for $1 < i \leq l/2$, and the projection of $Y_a(W_1)$ to $W_2$ is nonzero. Since the $W_i$ are pairwise nonisomorphic irreducible $J_a$ modules, we deduce the claim. 

Applying Proposition \ref{prop:classicaltool}, we see that for some $b \in \N$ we have $X_b = G$ and $B_bY_b$ is an epimorphic subgroup of $G$ of dimension three.   

\subsection{$G$ of type $A_l$ with $l \geq 3$ odd}\label{sec:Aodd}

Let $l=2m-1$ and $K \leq G$ be a Levi factor of type $A_1^m$ corresponding to the roots $\{\alpha_1, \alpha_3, \ldots, \alpha_{2m-1}\}$. The subgroup $K$ is contained in $M$, a maximal connected subgroup of $G$ of type $C_{m}$. We may take a set of simple roots $\beta_1, \ldots, \beta_m$ of $M$ such that the corresponding root groups $U_{\beta_i} = \{u_i(t) \mid t \in k \}$ are given in terms of root elements of $G$ as $u_i(t) = x_{ \alpha_{l-i} + \alpha_{l-i+1} }(t) x_{-\alpha_{l-i-1}-\alpha_{l-i}}(t)$ for $1 \leq i \leq m-1$ and $u_m(t) = x_{\alpha_1}(t)$.  

For each $a \in \N$ define $J_a$ to be a subgroup of type $A_1$ diagonally embedded in $K$ with twists $q_1 = 1, q_2 = p^a, \dots,q_{m} = p^{a(m-1)}$; this is conjugate to the subgroup $J_a$ defined in Section \ref{sec:C}. The subgroup $Y_a \leq M$ from the same argument becomes $Y_a = \{ y_a(t)z_a(t) \mid t \in k \}$ where 
\begin{align*}
y_a(t) & = x_{-\alpha_2}(t) x_{-\alpha_4}(t^{p^a}) \cdots x_{-\alpha_{l-1}}(t^{p^a(m-1)}), \\
z_a(t) & = x_{\alpha_1+\alpha_2+\alpha_3}(t) x_{\alpha_3+\alpha_4+\alpha_5}(t^{p^a}) \cdots x_{\alpha_{l-3}+\alpha_{l-2}+\alpha_{l-1}}(t^{p^{a(m-1)}}).
\end{align*}
Furthermore, it follows from Section \ref{sec:C} that there exists $b \in \N$ such that $M = \langle J_b, Y_b \rangle$. Let $Z_b =  \{ z_b(t) \mid t \in k \}$, a $1$-dimensional unipotent subgroup which commutes with $Y_b$. One checks that a Borel subgroup $B_b$ of $J_b$ normalises $Z_b$. Since the $q_i+q_{i+1}$ weight spaces of $L(M) \downarrow J$ are $1$-dimensional it follows that $Z_b$ is not contained in $M$. We conclude as usual, obtaining a $4$-dimensional epimorphic subgroup $B_bY_bZ_b$ of $G$. 

\subsection{$G$ of type $A_l$ with $l$ even and $p\neq 2$} \label{sec:Alevenpodd}

Let $m = (l-2)/2$ and $H \leq G$ be a Levi factor of type $A_1^{m} A_2$ corresponding to the roots $\{\alpha_1, \alpha_3, \ldots, \alpha_{l-3}, \alpha_{l-1}, \alpha_{l} \}$. Let $K \leq H$ be a subgroup of type $A_1^{m + 1}$, where the subgroup of type $A_1$ in the $A_2$-factor acts irreducibly on the natural module for $A_2$. Let $J_a\leq K$ be a subgroup of type $A_1$ diagonally embedded in $K$ with twists $q_1 = 1, q_2=p^a,\dots,q_{m+1} = p^{am}$, and chosen such that a Borel subgroup $B_a \leq J_a$ (resp. a maximal torus $T_a \leq B_a \leq J_a$) is contained in $B$ (resp. $T$). In particular, $W \downarrow J_a = W_1 \oplus \cdots \oplus W_m \oplus W_{m+1}$ where $W_i$ is irreducible of highest weight $q_i$ for $1\leq i\leq m$ and $W_{m+1}$ is irreducible of highest weight $2q_{m+1}$. Letting $W_0 = \oplus_{i=1}^{m} W_i$, we note that the projection of $J_a$ to $\textrm{Sp}(W_0)$ is conjugate to the subgroup called $J_a$ in Section \ref{sec:C}. 

A direct calculation shows that $U_{-\alpha_2-\cdots-\alpha_{l-2}}$ fixes the subspace $W_{m+1}$ and the image of $W_0$ under the action of $U_{-\alpha_2-\cdots-\alpha_{l-2}}$ has nonzero projection in $W_{m+1}$. The root group $U_{\alpha_0}$ does the opposite; to be precise, $U_{\alpha_0}$ fixes the subspace $W_0$ and the image of $W_{m+1}$ under the action of $U_{\alpha_0}$ has nonzero projection in $W_0$. Let  $Z_a = \{ x_{-\alpha_2-\cdots-\alpha_{l-2}}(t) x_{\alpha_0}(t) \mid t \in k \}$, a $1$-dimensional unipotent group on which $T_{a}$ acts via weight $q_1+2q_{m+1}$. A direct calculation shows that $Z_a$ is normalised by the positive root groups of $K$. 

Now let \[Y_a = \{ x_{-\alpha_2}(t) x_{\alpha_1+\alpha_2+\alpha_3}(t) x_{-\alpha_4}(t^{p^a}) x_{\alpha_3+\alpha_4+\alpha_5}(t^{p^a}) \cdots \]
\[x_{-\alpha_{l-3}}(t^{p^a(m-1)}) x_{\alpha_{l-3}+\alpha_{l-2}+\alpha_{l-1}}(t^{p^{a(m-1)}}) \mid t \in k \}.\] 
One checks that $Y_a$ is a $1$-dimensional unipotent group normalised by the positive root groups of $K$, and has $T_a$-weight $1+p^a$. Therefore $Y_a$ is normalised by $B_a$. Furthermore, $Y_a$ commutes with $Z_a$. This subgroup $Y_a$ is also contained in $\textrm{Sp}(W_0)$ and is conjugate to the subgroup called $Y_a$ in Section \ref{sec:C}. Combining this with the remarks in the previous paragraph we conclude that $X_a = \langle J_a,Y_a,Z_a\rangle$ acts irreducibly on $W$. Now Proposition \ref{prop:classicaltool} applies, yielding a $4$-dimensional epimorphic subgroup of $G$. 

\subsection{$G$ of type $A_l$ with $l$ even and $p=2$} \label{sec:Alevenp2}

Let $M \leq G$ be a maximal subgroup of type $B_{l/2}$, the centraliser of a standard graph automorphism of $G$. We rely on the construction in Section~\ref{sec:C} and an isogeny $\pi:C_{l/2} \to M$. Let $\pi(J_b)\leq M$ be a subgroup of type $A_1$ and $\pi(Y_b)$ a 1-dimensional subgroup such that $M = \langle \pi(J_b), \pi(Y_b)\rangle$. We may assume that a Borel subgroup of $\pi(J_b)$ is contained in $B$. Note that the $T$-root group $U_{\alpha_0}$ is normalised by the Borel subgroup $B$ of $G$ and does not lie in the subgroup $M$. We deduce that $\langle \pi(J_b),\pi(Y_b), U_{\alpha_0}\rangle = G$ and therefore $\pi(B_b)\pi(Y_b)U_{\alpha_0}$ is a 4-dimensional epimorphic subgroup in $G$.

\subsection{$G$ of type $B_l$ with $l$ even and $p \neq 2$} \label{sec:Blevenpodd}

Let $M\leq G$ be a maximal subgroup of type $D_l$ containing our fixed choice of maximal torus $T$ of $G$. For each $a\in\N$, let $J_a \leq M$ be the subgroup of type $A_1$ and $Y_a$ the associated $1$-dimensional unipotent subgroup of $M$ as defined in Section~\ref{sec:Deven}. Furthermore, let $b\in\N$ be such that $M = \langle J_b,Y_b\rangle$. 

The action of $M$ on $L(G)$ is $$L(G)\downarrow M = L(M) \oplus \lambda_1,$$ and there exists $\gamma \in \Phi$ such that $e_\gamma$ is a highest weight vector for the summand $V_M(\lambda_1)$ (so $e_\gamma\not\in L(M)$). Restricting from $M$ to $J_b$ we see that
$$L(G)\downarrow J_b = L(M)\downarrow J_b \oplus \left(\bigoplus_{i=1}^{l-1} (q_i + q_{i+1})\right)$$ 
and $e_\gamma$ is a $T_b$-weight vector of weight $q_i+q_{i+1}$ for some $1 \leq i \leq l-1$. Defining $X_b = \langle J_b,Y_b,U_\gamma\rangle$, we see that $X_b$ properly contains the maximal connected subgroup $M$, and hence $X_b = G$. As usual, $B_bY_bU_\gamma$ is a 4-dimensional epimorphic subgroup of $G$.

\subsection{$G$ of type $B_l$ with $l$ odd and $p \neq 2$} \label{sec:Bloddpodd}

Let $K \leq G$ be a subsystem subgroup of type $A_1^l$ corresponding to the natural subgroup $\text{SO}_4^{(l-1)/2} \text{SO}_3 \leq G$. Let $J_a$ be a subgroup of type $A_1$ diagonally embedded in $K$ with twists $q_1 = 1, q_2 = p^a, q_3 = p^{2a} \ldots, q_l = p^{(l-1)a}$ and chosen such that a Borel subgroup $B_a \leq J_a$ (resp. a maximal torus $T_a \leq B_a \leq J_a$) is contained in $B$ (resp. $T$). Then  $W\downarrow J_a = W_1 \oplus \cdots \oplus W_{(l-1)/2} \oplus W_{(l+1)/2}$, where $W_i$ is irreducible of highest weight $q_{2i-1} + q_{2i}$ for $1 \leq i < (l+1)/2$ and $W_{(l+1)/2}$ is irreducible of highest weight $2q_l$. 

The action of $J_a$ on $L(G)$ has direct summands with highest weight vectors of weight $q_i+q_{i+1}+q_{j}+q_{j+1}$ for $i$ and $j$ odd,
distinct and less than or equal to $(l-1)/2$, and $q_i+q_{i+1}+2q_{l}$ for $i$ odd and less than or equal to $(l-1)/2$. All nonzero weights are afforded by $T$-root vectors in $L(G)$, and one can check that the associated roots are long roots which we call $\gamma_{i,j}\in\Phi$ (including $j=l$ for the second set of summands).  

For all $i,j,m,n$ odd, a consideration of $T_a$-weights shows that the root groups $U_{\gamma_{i,j}}$ and  $U_{\gamma_{m,n}}$ commute. Let $$Y_a = \{x_{\gamma_{1,3}}(t) x_{\gamma_{3,5}}(t^{p^3 a}) \cdots x_{\gamma_{l-3,l-1}}(t^{p^{(l-3)a}}) \mid t\in k\},\mbox { and } Z = U_{\gamma_{1,l}}.$$
Together these generate a $2$-dimensional unipotent subgroup of $G$, normalised by $T_a$. 

We now show that $X_a = \langle J_a,Y_a,Z\rangle$ acts irreducibly on $W$. Considering the $T_a$-weight spaces of $W$, we see that for $j\ne l$ the root group $U_{\gamma_{i,j}}$ fixes every weight space apart from the $2$-dimensional space spanned by vectors of weights
$-q_{i} - q_{i-1}$ and $-q_{j} - q_{j-1}$, and when $j=l$ the root group fixes every weight space apart from those for the weights  $-q_{i} - q_{i-1}$ and $-2q_{l}$.  Moreover, since $x_{\gamma_{i,j}}(t)$ is a long root element of $G$ (for $t\ne 0$), it acts on $W$ with Jordan blocks $J_2^2 + J_1^{2l-3}$. Combining these two observations we see that the projection of $Y_a(W_j)$ to $W_{j-1}$ is nonzero for $1 < j \leq (l-1)/2$ and the projection of $Y_a(W_{i})$ to $W_{i+1}$ is nonzero for $1 \leq i < (l-1)/2$. Furthermore, $Z$ sends some nonzero vector of $W_{(l+1)/2}$ to a nonzero vector of $W_1$. Since any $X_a$-submodule of $W$ is a sum of a subset of the $W_i$, we see that $X_a$ acts irreducibly on $W$.

We now complete the argument as in previous cases obtaining a $4$-dimensional epimorphic subgroup of $G$.  

\subsection{$G$ of type $D_l$ with $l \geq 5$ odd}\label{sec:Dlodd}

Let $P=QL \leq G$ be the standard maximal parabolic subgroup with Levi factor $L$ of type $D_{l-1}T_1$. The derived subgroup $H = [L,L]$ being of type $D_{l-1}$ allows us to use the construction from Section~\ref{sec:Deven}. In particular, there exists $J_b \leq H$ of type $A_1$ and $Y_b \leq H$ a $1$-dimensional unipotent group such that $H = \langle J_b, Y_b \rangle$ (chosen such that a Borel subgroup $B_b$ of $J_b$ and the subgroup $Y_b$ are contained in the standard Borel subgroup of $H$). 

It follows from \cite{ABS} that $L(Q)$ is an $H$-submodule of $L(G)$ isomorphic to $V_H(\lambda_1)$ with highest weight vector $e_{\alpha_0}$. Similarly, if $Q^{\textrm{op}}$ is the unipotent radical of the opposite parabolic containing $L$ then $L(Q^{\textrm{op}})$ is also an $H$-submodule of $L(G)$ isomorphic to $V_H(\lambda_1)$ with highest weight vector $e_{-\alpha_1}$. Note that the root groups $U_{-\alpha_1}$ and $U_{\alpha_0}$ commute. 

Let $X = \langle J_b, Y_b, U_{-\alpha_1}, U_{\alpha_0} \rangle$. It follows from the previous paragraph that $L(X)$ contains $L(H)$, $L(Q)$ and $L(Q^{\textrm{op}})$. Therefore $\dim(L(X)) \geq \dim(L(G)) - 1 $ and so $\dim(X) \geq \dim(G) - 1$. Clearly, by dimension considerations $X$ is not contained in a parabolic subgroup. Since $X$ has rank at least $l-1$, it follows from \cite[Theorem~1]{LSTlargerank} that if $X$ is proper then it is either of type $B_{l-1}$ or a semisimple subsystem subgroup of maximal rank. These are also ruled out by dimension considerations. Therefore Proposition~\ref{prop:BBmethod} yields a $5$-dimensional epimorphic subgroup of $G$, namely $B_b Y_b U_{-\alpha_1} U_{\alpha_0}$.

\begin{rem}
When $l \equiv 3 \pmod 6$ and $p \neq 2$ one can construct a $3$-dimensional epimorphic subgroup with our usual arguments. Indeed, there is a family of subgroups $J_a$ of type $A_1$ diagonally embedded in a subgroup of type $A_1^{\frac{2l}{3}}$ (the natural $\textrm{SO}_3^{\frac{2l}{3}}$ subgroup of $\textrm{SO}_{2l}$) that have an appropriate family of $1$-dimensional unipotent subgroups $Y_a$. 
\end{rem}

\section{Proof of Theorem \ref{mainthm} for $G$ of exceptional type} \label{sec:exceptional}
 
Let $G$ be a simple algebraic group of exceptional type. A subgroup is $G$-irreducible if it is not contained in a proper parabolic subgroup of $G$. The classification of all connected $G$-irreducible subgroups is contained in \cite{ThoMem}; we use this extensively to show certain subgroups are not contained in any maximal connected subgroup of $G$.   

\subsection{$G$ of type $F_4$}

\subsubsection{Characteristic $p \neq 2$}

Let $M \leq G$ be a maximal connected subgroup of type $B_4$ with $9$-dimensional natural module $V$. The group $M$ has a maximal connected subgroup $K$ of type $A_1^2$ acting on $V$ as $(2,2)$. Define $J \leq K$ to be a diagonal subgroup embedded with two distinct twists $1,q$. In \cite{ThoMem}, it is shown that $J$ is $G$-irreducible (denoted by $F_4(\#5)$). Furthermore, one uses Theorem~1 and the explanation in Chapter~11 of \textit{ibid.} to deduce the conjugacy classes of connected overgroups of $J$. In this case, since Table 11.2A does not have a row for the class of subgroups $F_4(\#5)$ we see that the only connected overgroups are of type $A_1^2$ and conjugate to $K$ (the class of $K$ is denoted $F_4(\#29)$). Table 11.2A does contain $F_4(\#29)$ and tells us that the only connected overgroup of $K$ is of type $B_4$. Thus, all maximal connected subgroups containing $J$ are of type $B_4$. 

The action of $M$ on the module $V_G(\lambda_4)$ is 
$$ V_G(\lambda_4) \downarrow M =  \lambda_1 \oplus \lambda_4 \oplus 0^{1-\delta_{p,3}}.$$ 
Restricting to $K$ and then to $J$ we find that 
$$ V_G(\lambda_4) \downarrow J = 2 \otimes 2^{[q]} \oplus W(3) \otimes 1^{[q]} \oplus W(3)^{[q]} \otimes 1 \oplus 0^{1-\delta_{p,3}}.$$ 
(Note that the composition factors for the action of $K$ are given in \cite[Table~12.2]{ThoMem} but one has to do a little more work to obtain the exact restriction.)

We now see that $M$ and $J$ both fix a unique $10$-dimensional summand $N$ of $V_G(\lambda_4)$ when $p > 3$ and a unique $9$-dimensional summand $N$ of $V_G(\lambda_4)$ when $p=3$ . Since $M$ is a maximal subgroup of $G$, the full stabiliser in $G$ of $N$ is $M$. Therefore, it follows that $M$ is the only maximal overgroup of $J$. 

To complete this case we consider the action of $M$ on $L(G)$, which is  
\[ L(G) \downarrow M = \lambda_2 \oplus \lambda_4.\]
Since $M$ is maximal rank, there exists a root $\gamma \in \Phi$ such that $e_\gamma$ is a highest weight vector for the $M$-summand $\lambda_4$. Let $Y = U_\gamma$ and note that $Y$ is not contained in $M$. By construction, a Borel subgroup $B_M$ of $M$ normalises $Y$. Choosing a Borel subgroup $B_J$ of $J$ that is contained in $B_M$, we see that a Borel subgroup of $J$ also normalises $Y$. Let $X = \langle J, Y\rangle$. Since $J$ is a subgroup of $X$ the only possibility for a maximal connected overgroup of $X$ is $M$. But $M$ does not contain $Y$ and therefore we conclude that $X = G$. Therefore, Proposition \ref{prop:BBmethod} implies that $B_J J$ is a $3$-dimensional epimorphic subgroup of $G$. 

\subsubsection{Characteristic $p = 2$} \label{s:F4p2}

Let $K \leq G$ be the subsystem subgroup of type $A_1^2 \tilde{A}_1^2$ corresponding to the root groups $U_{\pm \beta_i}$ for $\beta_1 = \alpha_2$, $\beta_2 = 0120$, $\beta_3 = 1110$, $\beta_4 = 1232$. Then $K$ is a subgroup of the subsystem subgroup $H$ of type $B_2^2$ corresponding to the simple roots $\alpha_2, \alpha_3, 0122, 1110$. One also checks that $H$ is contained in the subsystem subgroups $M_1$ of type $B_4$, and $M_2$ of type $C_4$, with simple roots $-2342, \alpha_1, \alpha_2, \alpha_3$ and $\alpha_2,\alpha_3,\alpha_4, -1232$, respectively. Note that $M_1$ and $M_2$ are interchanged by the standard graph morphism of $G$ but $K$ and $H$ are stabilised. 

Define $J \leq K$ to be a diagonal subgroup embedded via the twists $q_1 = 4$, $q_2 = 32$, $q_3 = 1$, $q_4 = 8$. From \cite{ThoMem}, we find that the subgroup $J$ is $G$-irreducible (denoted $F_4(\#2)$). Table 11.2A in \textit{ibid.} shows that every connected overgroup of $J$ contains a conjugate of $K$ and thus every maximal connected overgroup of $J$ is of type $B_4$ or $C_4$. Furthermore, using that same table we deduce that if $J$ is contained in a subgroup $M$ of type $B_4$ or $C_4$ then the action of $J$ on the natural module of $M$ is the same as the action on the natural module of $M_1$, $M_2$, respectively. In particular, $J$ is contained in a subgroup of type $B_2^2 \leq M$ in both cases. 

Next, we prove that $M_1$ is the unique maximal connected overgroup of type $B_4$ for both $H$ and $J$, from which it follows that $H$ is the unique overgroup of type $B_2^2$ of $J$. To see this we consider the actions of $M_1, H$ and $K$ on $V_{26} = V_{G}(\lambda_4)$. 
\begin{align*} 
& V_{26} \downarrow M_1  = 0 | \lambda_1 | 0 \oplus \lambda_4, \quad V_{26} \downarrow H = 0|( (10,0) \oplus (0,10) )|0 \oplus (01,01), \\
& V_{26} \downarrow K  = (1,1,0,0) \oplus 0|((0,0,2,0) \oplus (0,0,0,2) )|0  \oplus (1,0,1,1) \oplus (0,1,1,1).
\end{align*}
Immediately we see that $M_1$ is the stabiliser in $G$ of the $16$-dimensional summand $V_{M}(\lambda_4)$ of $V_{26}$. Since $H \leq M_1$ stabilises a unique $16$-dimensional summand of $V_{26}$, the only overgroup of $H$ of type $B_4$ is $M_1$. For the second part, we recall that whenever $J$ is contained in a subgroup of type $B_4$ we know its action on the natural module. Its action on $V_{B_4}(\lambda_4)$ is thus determined and is $(1 \otimes 1^{[4]} \otimes 1^{[8]}) \oplus (1 \otimes 1^{[8]} \otimes 1^{[32]})$. We see that $J$ stabilises a unique $16$-dimensional summand of $V_{26}$ with this action so the only overgroup of $J$ of type $B_4$ is $M_1$. 

As mentioned above, $H$ is fixed under the standard graph morphism of $G$ so it follows that $M_2$ is the unique maximal connected overgroup of $H$ of type $C_4$. We then conclude that $M_2$ is the unique overgroup of $J$ of type $C_4$, since $J$ is always contained in an intermediate subgroup of type $B_2^2$ when contained in a subgroup of type $C_4$. 

Let $\gamma_1 = 0110, \gamma_2 = 2342, \gamma_3 = 1221, \gamma_4 = 1231, \gamma_5 = 1220, \gamma_6 = 1342$. One calculates that the roots $\gamma_1, \ldots, \gamma_6$ afford the weights $(1,1,0,0)$, $(0,0,2,2)$, $(1,0,1,1)$, $(0,1,1,1)$, $(1,1,2,0)$, $(1,1,0,2)$, respectively. Furthermore, the root groups $U_{\beta_i}$ for $1 \leq i \leq 4$ centralise each root group $U_{\gamma_j}$ for $1 \leq j \leq 6$. The roots $\gamma_1$ and $\gamma_2$ are both in the root subsystems corresponding to $M_1$ and $M_2$. The roots $\gamma_3$ and $\gamma_4$ are not in the root subsystem corresponding to $M_1$, whereas $\gamma_5$ and $\gamma_6$ are not in the root subsystem corresponding to $M_2$. 

Define $y(t) = x_{\gamma_3}(t) x_{\gamma_6}(t^4)$ and $Y = \{ y(t) \mid t \in k \}$. Then $Y$ is $1$-dimensional since the root groups $U_{\gamma_3}$ and $U_{\gamma_6}$ commute, and $Y$ is not a subgroup of either $M_1$ or $M_2$. By construction, there is a Borel subgroup $B_J$ of $J$ normalising $Y$. Since $M_1$ and $M_2$ are the only maximal connected overgroups of $J$ it follows that $X = \langle J, Y \rangle$ is not contained in any maximal connected subgroup of $G$; thus $X = G$, as required. Therefore, Proposition \ref{prop:BBmethod} implies that $B_J Y$ is a $3$-dimensional epimorphic subgroup of $G$. 

\subsection{$G$ of type $E_6$}

\subsubsection{Characteristic $p \neq 2$}

Let $M$ be a maximal rank subgroup of $G$ of type $A_2^3$. Since $p \neq 2$, each $A_2$-factor has an $A_2$-irreducible subgroup of type $A_1$. Let $K$ be the subgroup of type $A_1^3 \leq M$ where each $A_1$-factor is $A_2$-irreducible. Define a diagonal subgroup $J$ of type $A_1$ embedded via distinct twists $q_1,q_2,q_3$. In \cite{ThoMem}, $J$ is shown to be $G$-irreducible (denoted $E_6(\#3)$) and one deduces that any maximal connected overgroup of $J$ has type $A_2^3$ or $C_4$. 

The action of $M$ on $L(G)$ is as follows 
\[ L(G) \downarrow M = L(A_2^3) \oplus (10,10,10) \oplus (01,01,01).\]
In particular, there exist roots $\gamma_1, \gamma_2$ such that $e_{\gamma_i}$ are highest weight vectors of the final two summands. The root groups $U_{\gamma_1}$ and $U_{\gamma_2}$ commute since $(11,11,11)$ is not a weight for $M$ on $L(G)$. Restricting the action of $M$ on $L(G)$ to $J$ yields
\[ L(G) \downarrow J = 2^{[q_1]} \oplus 2^{[q_2]} \oplus 2^{[q_3]} \oplus W(4)^{[q_1]} \oplus W(4)^{[q_2]} \oplus W(4)^{[q_3]} \oplus (2^{[q_1]} \otimes 2^{[q_2]} \otimes 2^{[q_3]})^2.\]
 
Define $Y =\langle U_{\gamma_1}, U_{\gamma_2} \rangle$. By construction, a Borel subgroup of $M$ and thus a Borel subgroup $B_J$ of $J$, normalises $Y$. The Lie algebra of $X = \langle J,Y \rangle$ contains $e_{\gamma_1}$, $e_{\gamma_2}$ and is a submodule of $L(G)$ under the action of $X$, so in particular under the action of $J$. Thus, it contains the two $27$-dimensional summands with highest weight $2^{[q_1]} \otimes 2^{[q_2]} \otimes 2^{[q_3]}$. Since the only maximal connected overgroups of $J$ have dimension at most $36$, it follows that $X = G$. Therefore, Proposition \ref{prop:BBmethod} implies that $B_J Y$ is a $4$-dimensional epimorphic subgroup of $G$.

\subsubsection{Characteristic $p = 2$} 

Let $J$, $y(t)$ and $Y$ be as in Section~\ref{s:F4p2}. Considering the maximal subgroup $M \leq G$ of type $F_4$ as the centraliser of a standard graph automorphism of $G$, the relevant root elements of $F_4$ expressed as unipotent elements of $E_6$ are
\begin{align*}
x_{\beta_1}(t) = x_{000100}(t), & \quad x_{\beta_2}(t)  = x_{001110}(t), \\
x_{\beta_3}(t) = x_{112211}(t) x_{111221}(t), & \quad x_{\beta_4}(t)  = x_{011100}(t) x_{010110}(t),\\x_{\gamma_3}(t) = x_{111210}(t) x_{011211}(t), & \quad x_{\gamma_6}(t) = x_{112321}(t).
\end{align*}  
A straightforward calculation shows that a Borel subgroup $B_J$ of $J$ normalises $Z = U_{011221}$ and $Y$ commutes with $Z$. Since $M = \langle J,Y \rangle$ is a maximal subgroup of $G$, it follows that $G = \langle J, Y, Z \rangle$. Therefore, Proposition \ref{prop:BBmethod} implies that $B_J Y Z$ is a $4$-dimensional epimorphic subgroup of $G$. 

\subsection{$G$ of type $E_7$}

Let $H \leq G$ be a maximal connected subgroup of type $A_1D_6$ and let $V$ be the $12$-dimensional natural module for the $D_6$-factor.  

\subsubsection{Characteristic $p \neq 2$}

Let $K$ be the subgroup of type $A_1^5 \leq H$ such that the projection of $K$ to $D_6$ acts as $(2,0,0,0) \oplus (0,2,0,0) \oplus (0,0,2,0) \oplus (0,0,0,2)$ on $V$. Define a diagonal subgroup $J$ of type $A_1$ embedded in $K$ via distinct twists $q_1, \ldots, q_5$. In \cite{ThoMem}, this subgroup is proved to be $G$-irreducible (denoted $E_7(\#11)$) and one deduces that every maximal connected overgroup of $J$ is of type $A_1 D_6$. 


By \cite[Lemma~11.8]{LSbook}, the action of $M$ on $V_G(\lambda_7)$ is
\[ V_G(\lambda_7) \downarrow M =  (1,\lambda_1) \oplus (0,\lambda_6),\] and it follows that
\[ V_G(\lambda_7) \downarrow J = 1^{[q_1]} \otimes 2^{[q_2]} \oplus 1^{[q_1]} \otimes 2^{[q_3]} \oplus 1^{[q_1]} \otimes 2^{[q_4]} \oplus 1^{[q_1]} \otimes 2^{[q_5]} \oplus (1^{[q_2]} \otimes \cdots \otimes 1^{[q_5]})^2.\] 

Since $M$ is maximal, it is the stabiliser in $G$ of the $32$-dimensional summand $(0,\lambda_6)$. Looking at $V_G(\lambda_7) \downarrow J$ we find that $J$ stabilises a unique summand of dimension $32$ on $V_G(\lambda_7)$ and hence $M$ is the unique connected maximal overgroup of $J$. 

By \textit{loc. cit.}, \[ L(G) \downarrow M = L(M) \oplus (1, \lambda_5),\]
and we let $e_\gamma$ be a highest weight vector for the summand $(1, \lambda_5)$. Define $Y = U_\gamma$ and note that $Y$ is not contained in $M$. By construction, a Borel subgroup of $M$ normalises $Y$ and so we may assume a Borel subgroup $B_J$ of $J$ normalises $Y$. Since $M$ is the unique maximal overgroup of $J$, it follows that $X = \langle J, Y \rangle$ is not contained in any maximal connected subgroup of $G$. Thus $X = G$, and $B_J Y$ is a $3$-dimensional epimorphic subgroup of $G$.  

\subsubsection{Characteristic $p = 2$}

Let $K$ be the subgroup of type $A_1^6 \leq H$ such that the projection of $K$ to $D_6$ acts as $0 | ( (2,0,0,0,0) \oplus \cdots \oplus (0,0,0,0,2)) | 0$ on $V$. Define a diagonal subgroup $J$ of type $A_1$ embedded in $K$ via distinct twists $q_1, \ldots, q_6$. In \cite{ThoMem}, this subgroup is shown to be $G$-irreducible (denoted $E_7(\#14)$) and as before one uses the tables therein to  deduce that any maximal connected overgroup of $J$ is of type $A_1 D_6$. 

It follows as in the previous case that $M$ is the unique maximal connected overgroup of $J$. The action of $J$ on the $M$-summand $(1, \lambda_6)$ is $1^{[q_1]} \otimes \cdots \otimes 1^{[q_6]}$, and we let $e_{\gamma}$ be a highest weight vector. A Borel subgroup $B_J$ of $J$ normalises $Y = U_\gamma$ and $X = \langle J, Y \rangle$ is again the whole of $G$. Thus $B_J Y$ is a $3$-dimensional epimorphic subgroup of $G$.  

\subsection{$G$ of type $E_8$}

Let $H \leq G$ be a maximal connected subgroup of type $D_8$ with $16$-dimensional natural module $V$.  

\subsubsection{Characteristic $p \neq 2$}

Let $K \leq H$ be the subgroup of type $A_1^6$ acting as $(1,1,0,0,0,0) \oplus (0,0,2,0,0,0) \oplus \cdots \oplus (0,0,0,0,0,2)$ on $V$. Define a diagonal subgroup $J$ of type $A_1$ embedded in $K$ via the twists $q_1 = 1$, $q_2 = p^5$, $q_3 = p$, $q_4 = p^2$, $q_5 = p^3$, $q_6 = p^4$. In \cite{ThoMem}, this subgroup is shown to be $G$-irreducible (denoted $E_8(\#25)$). One deduces from Table~11.5A that any maximal connected overgroup of $J$ is either of type $D_8$ or $A_1 E_7$. 

Note that $K$ is contained in a maximal rank subgroup $A_1^2 A_3^2$ of $H$ which acts on $V$ as $(1,1,000,000) \oplus (0,0,010,000) \oplus (0,0,000,010)$ and 
\[ L(G) \downarrow A_1^2 A_3^2 = L(A_1^2 A_3^2) \oplus (1,1,010,0) \oplus (1,1,0,010) \oplus (0,0,010,010) \oplus \] \[(1,0,100,100) \oplus (1,0,001,001) \oplus (0,1,100,100) \oplus (0,1,001,001).\] 
In particular, there exist four roots $\gamma_1, \gamma_2, \gamma_3, \gamma_4$ such that $e_{\gamma_i}$ is a highest weight vector for the final four summands, respectively. Since the sum of any two of these weights is not a weight for the action of $A_1^2 A_3^2$ on $L(G)$, the corresponding root group elements $x_{\gamma_1}(t), \ldots, x_{\gamma_4}(t)$ pairwise commute.  

Restricting from $A_1^2 A_3^2$ to $K$ and then to $J$ it follows that 
\[ L(G) \downarrow J = 2^{[q_1]} \oplus \cdots \oplus 2^{[q_6]} \oplus 2^{[q_3]} \otimes 2^{[q_4]} \oplus 2^{[q_3]} \otimes 2^{[q_5]} \oplus \cdots \oplus 2^{[q_5]} \otimes 2^{[q_6]} \oplus \] \[ (1^{[q_1]} \otimes 1^{[q_2]} \otimes 2^{[q_3]}) \oplus \cdots \oplus (1^{[q_1]} \otimes 1^{[q_2]} \otimes 2^{[q_6]}) \oplus \] \[ (1^{[q_1]} \otimes 1^{[q_3]} \otimes 1^{[q_4]} \otimes 1^{[q_5]} \otimes 1^{[q_6]})^2 \oplus (1^{[q_2]} \otimes 1^{[q_3]} \otimes 1^{[q_4]} \otimes 1^{[q_5]} \otimes 1^{[q_6]})^2,\]
with the final four summands having highest weight vectors $e_{\gamma_i}$. 

By construction, a Borel subgroup $B_J$ of $J$ normalises the $1$-dimensional unipotent group $Y = \{ x_{\gamma_1}(t)x_{\gamma_3}(t^p)  \mid t \in k \}$. Let $X = \langle J, Y \rangle$. We will show that no maximal connected overgroup of $J$ contains $Y$ so that $X = G$, finishing the proof with the usual argument to show that $B_J Y$ is a $3$-dimensional epimorphic subgroup of $G$.  

As mentioned above, any maximal connected overgroup of $J$ is either of type $D_8$ or $A_1E_7$. Using Table 11.5A in \cite{ThoMem}, we deduce that every $G$-conjugate of $J$ contained in $M$ is contained in a subgroup conjugate to $K$ (denoted $E_8(\#365)$), which itself is contained in a subgroup of type $A_1^2 D_6$ in $M$ (denoted $E_8(\#107)$).

Let $M$ be a maximal connected overgroup of $J$ of type $D_8$ so that \[L(G) \downarrow M = L(M) \oplus \lambda_7.\] A subgroup of type $A_1^2 D_6$ acts on $V_{D_8}(\lambda_7)$ as $(1,1,V_{D_6}(\lambda_5)) + (1,1,V_{D_6}(\lambda_6))$. Therefore the action of $J$ on $V_{D_8}(\lambda_7)$ is $(1^{[q_1]} \otimes 1^{[q_3]} \otimes 1^{[q_4]} \otimes 1^{[q_5]} \otimes 1^{[q_6]})^2 \oplus (1^{[q_2]} \otimes 1^{[q_3]} \otimes 1^{[q_4]} \otimes 1^{[q_5]} \otimes 1^{[q_6]})^2$. There is a unique such $128$-dimensional $J$-summand and so it follows that $M$ is the unique overgroup of type $D_8$ and $M = H$. By construction, $M$ does not contain $Y$ and so $X$ is not contained in $M$. 

Now suppose that $M$ is a maximal connected overgroup of $J$ of type $A_1E_7$ so \[L(G) \downarrow M = L(M) \oplus (1,\lambda_7).\] We will prove that there are exactly two ways to make complementary $J$-summands of $L(G)$ with dimensions $136$ and $112$, that are compatible with $J$ being a subgroup of $M$. This will prove that there are exactly two choices for $M$. The $M$-summand $(1,\lambda_7)$ restricts to ${A}_1^2 D_6$ as either $(1,1,\lambda_1) \oplus (1,0,\lambda_5)$ or $(1,1,\lambda_1) \oplus (1,0,\lambda_6)$. In both cases, the restriction to an appropriate subgroup of type $A_1^6$ (conjugate to $K$) is then $(1,1,2,0,0,0) \oplus \cdots \oplus (1,1,0,0,0,2) \oplus (1,0,1,1,1,1)^2$. Therefore, since $M$ contains $J$ it follows that both summands with highest weight vector $e_{\gamma_1}, e_{\gamma_2}$ are contained in $L(M)$ and both summands with highest weight vector $e_{\gamma_3}, e_{\gamma_4}$ are not, or vice versa. Furthermore, the four $J$-summands $(1^{[q_1]} \otimes 1^{[q_2]} \otimes 2^{[q_3]}), \ldots, (1^{[q_1]} \otimes 1^{[q_2]} \otimes 2^{[q_6]})$ come from $(1,\lambda_7) \downarrow J$. Therefore there are precisely two ways to separate the summands of $L(G) \downarrow J$ into complementary summands of dimensions $136$ and $112$ such that the first can be $L(M) \downarrow J$ and the second $(1,\lambda_7) \downarrow J$. Since $M$ is a maximal subgroup of $G$ it is the stabiliser of this decomposition into a pair of subspaces and so there are at most two overgroups of $J$ of type ${A}_1E_7$. Depending on which subgroup ${A}_1$ of $K$ one chooses to take as the ${A}_1$-factor of ${A}_1 E_7$, there are certainly at least two such subgroups. Hence there are exactly two. From the previous discussion, we see that one overgroup contains $U_{\gamma_1}$ but not $U_{\gamma_3}$ and the opposite is true of the other overgroup. Thus neither can contain $Y$, completing the proof of the claim.  

\subsubsection{Characteristic $p = 2$}

Since $p=2$, there exists a subgroup $K \leq H$ of type $A_1^7$ acting as $0 | ((2,0,0,0,0,0,0) \oplus \cdots \oplus (0,0,0,0,0,0,2)) | 0$ on $V$. Define $J$ of type $A_1$ to be a diagonal subgroup of $K$ with distinct field twists $q_1, \ldots, q_7$. In \cite{ThoMem}, $J$ is shown to be $G$-irreducible (denoted $E_8(\#30)$) and one deduces that all maximal connected overgroups of $J$ are of type $D_8$. 

As before, $L(G) \downarrow D_8 = L(D_8) \oplus V_{D_8}(\lambda_7)$. The action of $J$ on the half-spin module $V_{D_8}(\lambda_7)$ is $1^{[q_1]} \otimes \cdots \otimes 1^{[q_7]}$. In particular, the action of $J$ on $L(G)$ has a unique summand of dimension $128$ and it follows that there is a unique maximal connected overgroup of $J$, namely $H$. Let $e_{\gamma}$ be the highest weight vector of $V_{D_8}(\lambda_7)$. Then $U_\gamma$ is not contained in $H$ and so $X = \langle J, U_\gamma \rangle$ is $G$. Thus $B_J Y$ is a $3$-dimensional epimorphic subgroup of $G$ for any Borel subgroup $B_J \leq J$. 

\section*{Acknowledgements} 

The authors thank Prof. Michel Brion for helpful discussions about the algebraic geometry motivation for epimorphic subgroups and Prof. Gunter Malle for his suggestions on a preliminary version of this paper. 

This work was supported by Swiss National Science Foundation grant number 200020\_207730 (first author) and EPSRC grant EP/W000466/1 (second author).

For the purpose of open access, the authors have applied a Creative Commons Attribution (CC BY) licence to any Author Accepted Manuscript version arising from this submission.

\bibliographystyle{plain}
\bibliography{rsch}

\end{document}